\documentclass[11pt,reqno]{article}

\usepackage[T1]{fontenc}
\usepackage[margin=1in]{geometry}  
\usepackage{amsmath}               
\usepackage{amsfonts}              
\usepackage{amsthm}      
\usepackage{enumerate}
\usepackage{amssymb}
\usepackage{bbm}
\usepackage{comment}
\usepackage[utf8]{inputenc}
\usepackage{mathtools}

\usepackage{xcolor}   
\usepackage{hyperref}
\hypersetup{
    colorlinks=true, 
    linktoc=all,     
    linkcolor=black,  
}

\DeclarePairedDelimiter\abs{\lvert}{\rvert}%
\DeclarePairedDelimiter\norm{\lVert}{\rVert}%

\makeatletter
\let\oldabs\abs
\def\abs{\@ifstar{\oldabs}{\oldabs*}}
\let\oldnorm\norm
\def\norm{\@ifstar{\oldnorm}{\oldnorm*}}

\makeatother

\theoremstyle{definition}
\newtheorem{theorem}{Theorem}[section]
\newtheorem{lemma}[theorem]{Lemma}
\newtheorem{proposition}[theorem]{Proposition}

\newtheorem{corollary}[theorem]{Corollary}

\newtheorem{definition}[theorem]{Definition}

\newtheorem{remark}[theorem]{Remark}

\newtheorem{question}[theorem]{Question}
\numberwithin{equation}{section}

\DeclareMathOperator{\Aut}{Aut}

\DeclareMathOperator{\Sym}{Sym}
\DeclareMathOperator{\diam}{diam}

\DeclareMathOperator{\supp}{supp}
\DeclareMathOperator{\im}{im}

\let\phi\varphi

\newcommand{\eps}{\varepsilon}

\newcommand{\actson}{\curvearrowright}

\newcommand{\EE}{\mathbb{E}} 

\newcommand{\RR}{\mathbb{R}}      

\newcommand{\ZZ}{\mathbb{Z}}      
      
\newcommand{\PP}{\mathbb{P}}      
\newcommand{\FF}{\mathbb{F}} 

\newcommand{\NN}{\mathbb{N}}     

\newcommand{\one}{\mathbbm{1}}

\newcommand\restr[2]{{
  \left.\kern-\nulldelimiterspace 
  #1 
  \vphantom{\big|} 
  \right|_{#2} 
  }}

\renewcommand{\cal}[1]{{\mathcal #1}}

\newcommand\GG{\mathbb G_\bullet}
\newcommand\GGG{\mathbb G_{\bullet\bullet}}

\DeclareMathOperator{\ran}{Ran}

\DeclareMathOperator{\esup}{ess\,sup}

\DeclareMathOperator{\Leb}{Leb}

\title{Skeletons and Spectra: Bernoulli graphings are relatively Ramanujan}
\author{H\'ector Jard\'on-S\'anchez, L\'aszl\'o M\'arton T\'oth }
\date{\today}

\begin{document}

\maketitle
\begin{abstract}
    The aim of this paper is to investigate the spectral theory of unimodular random graphs and graphings representing them. We prove that Bernoulli graphings are relatively Ramanujan with respect to their skeleton Markov chain. That is, the part of their spectrum that comes from the random labels falls within the appropriate Alon-Boppana bound. This result complements an example due to Fr\k{a}czyk of an ergodic unimodular random graph with almost sure spectral gap but non-expanding Bernoulli graphing.

    We also highlight connections of our work with the theory of finite random graphs. Exploiting the result of Bordenave and Collins on random lifts being relatively almost Ramanujan, we prove a strengthening of our main theorem for unimodular quasi-transitive quasi-trees.

\end{abstract}
\setcounter{tocdepth}{2}
\tableofcontents

\newpage
\section{Introduction}

Measurable group theory and combinatorics often investigate properties of \emph{graphings}. These are bounded degree graphs with vertices $V(\mathcal{G})=X$, where $(X, \mu)$ is a standard Borel probability space, and a symmetric Borel set of edges $E(\mathcal{G}) \subseteq X \times X$ with the additional property that partial Borel bijections along edges of the graph preserve $\mu$. Graphings arise naturally from probability measure preserving actions of countable groups, and also serve as limits of sequences of sparse finite graphs in the sense of Benjamini--Schramm \cite{BSprocesses}. 

This paper investigates the interplay between the spectral properties of graphings and the \emph{unimodular random graphs} obtained by considering the rooted connected component of a $\mu$-random vertex of $\mathcal{G}$. We denote the law of this (rooted) random graph by $\nu_{\mathcal{G}}$, and denote a random instance by $(G,o)$. 
We start by stating our main contribution and continue to introduce the relevant notions as well as to motivate the result.

\begin{theorem} \label{thm:intro.rel.ram}
    The Bernoulli graphing $\mathcal{B}$ of an ergodic unimodular random graph $(G,o)$ is Ramanujan relative to its skeleton. That is $\sigma_R(\mathcal{B}) \subseteq [-\rho(G,o), \rho(G,o)]$, where $\sigma_R$ denotes the random part of the spectrum. 
\end{theorem}

\subsection{Bernoulli graphings and Ramanujan graphs}

For a finite graph $G$, we denote the absolute value of the second largest eigenvalue of its Markov operator $M$ by $\rho(G)$. 
When $G$ is a countably infinite, locally finite, connected graph, since the constant 1 vector is not in $\ell^2(V)$, there is no need to remove the trivial eigenvalue. In this case we set $\rho(G)=||M||_2$.
This quantity, called the \emph{spectral radius} of $G$, measures the exponential decay rate of the return probabilites $p_{2n}(o,o)$ of the simple random walk started from $o$. 
It is a well-known result of Kesten that in the infinite case $\rho(G)=1$ if and only if $G$ is amenable \cite{kesten1959full, kesten1959symmetric}. 

For a graphing $\mathcal{G}=(X,E,\mu)$ we write $\mathcal{M}$ for the Markov operator acting on $L^2(X)$. 
Analogously to the finite case, the spectral radius $\rho(\mathcal{G})$ is defined as the norm of $\mathcal{M}$ on the direct complement of the constant functions in $L^2(X)$. 
When $(G,o) \sim \nu_{\cal G}$, we refer to $\rho(G,o)$ as the \emph{local}, and $\rho(\cal G)$ as the \emph{global} spectral radius. When the graph or graphing is not regular, one must consider a degree-biased $L^2$-inner product in order to make the Markov operator self-adjoint, see~\ref{subsec:markov_operators}. We implicitly assume this throughout the introduction, and not burden the notation. 

It is natural to ask how $\rho(\mathcal{G})$ and $\rho(G,o)$ are related. Observe that $\rho(G,o)$ does not depend on the root $o$, so assuming ergodicity, $\rho(G,o)$ is a constant almost surely. We begin by showing that the global spectral radius is always bounded from below by the local one.

\begin{theorem}\label{thm:Alon-Boppana_intro}
    The inequality $\rho(G,o) \leq \rho(\mathcal{G})$ holds for all aperiodic graphings $\cal G$.
\end{theorem}

This result can be understood as an infinite version of the classic Alon-Boppana theorem, which we will consider more closely in the next subsection, and recall as Theorem~\ref{thm:alon_boppana} in section~\ref{subsubsec:ramanujan_graphs_and_graphings}. In fact, we establish a stronger statement, showing the containment of the entire spectrum in Corollary~\ref{cor:alon.boppana}. 

One cannot hope for an equality in Theorem \ref{thm:Alon-Boppana_intro} in full generality, as probability measure preserving actions of non-amenable groups need not have spectral gap. Nevertheless $\rho(G,o)=\rho(\mathcal{B})$ holds for \emph{Bernoulli graphings} $\mathcal B$ over Cayley graphs of groups \cite[Cor. 2.2]{lyons2011perfect} or regular trees \cite{backhausz2015ramanujan}. The Bernoulli graphing is constructed by adding iid uniform random $[0,1]$-labels to the vertices of a unimodular random graph $(G,o)$, and defining a graph on the set of such labelled configurations by connecting two if one can be obtained from the other by moving the root to a neighbor. The Bernoulli graphing is the appropriate version of the classical Bernoulli shift from symbolic dynamics in our setting. We give a precise definition in Section \ref{subsubsec:bernoulli_graphing}.

Following the terminology of \cite{backhausz2015ramanujan}, we say that $\mathcal{G}$ is \emph{Ramanujan} if $\rho(G,o)=\rho(\mathcal{G})$ holds. In \cite{bencs2024factor_nonamen}, it is proven that $\rho(G,o)<1$ implies $\rho(\mathcal{B}) < 1$ for Bernoulli graphings over quasi-transitive graphs, a weak version of the Ramanujan property. There it is asked whether this property holds for all unimodular random graphs. Here we present an unpublished counterexample due to Fr{\k{a}}czyk. 

\begin{theorem}[Fr{\k{a}}czyk]\label{thm:mikolay_intro}
    There exists an ergodic unimodular random graph with $\rho(G,o) < 1$ such that its Bernoulli graphing $\cal B$ is not strongly ergodic. In particular, $\rho(\cal B)=1$.
\end{theorem}

To construct this counterexample, a non-expanding action of a non-amenable group is encoded in the graph-theoretic structure of a unimodular random graph $(G,o)$ with $\rho (G,o) < 1$. Theorem~\ref{thm:intro.rel.ram} complements Theorem~\ref{thm:mikolay_intro} by showing that such an encoding is the only obstruction to Bernoulli graphings being Ramanujan. Precisely, this is the fact that Bernoulli graphings are \emph{relatively Ramanujan}. 

Let us now discuss the relative Ramanujan notion in more detail. Recall that the vertices of the Bernoulli graphing $\mathcal{B}$ are of the form $(G,\xi,o)$, where $\xi$ is iid uniform $[0,1]$-labeling of $V(G)$. The \emph{structured subspace} $S$ of  $L^2(V(\mathcal{B}))$ consists of those functions $f \in L^2 (V(\mathcal B))$ whose value $f(G,\xi,o)$ only depends on the graph structure $(G,o)$, and not on $\xi$. The orthogonal complement of $S$ is the \emph{random subspace} $R$. Both $S$ and $R$ are $\mathcal M$-invariant, so the spectrum of $\mathcal{M}$ splits accordingly into a \emph{structured} and \emph{random} part, denoted $\sigma_S$ and $\sigma_R$. In Section \ref{subsec:structured_spectrum}, we show that $\sigma_S$ is exactly the spectrum of the \emph{skeleton Markov chain} of $(G,o)$. Theorem \ref{thm:intro.rel.ram} shows that $\sigma_R$ satisfies the Ramanujan bound, hence the phrasing ``Ramanujan relative to its skeleton''. The notions discussed in this paragraph are introduced in detail in Sections~\ref{subsec:random_and_structured} and \ref{subsec:structured_spectrum}. 

Theorem~\ref{thm:intro.rel.ram} generalizes both of the Ramanujan graphing results mentioned before from \cite{backhausz2015ramanujan} and \cite{bencs2024factor_nonamen}. It also complements the general $\rho(G,o)\leq \rho(\cal G)$ result in the special case of Benoulli graphings. We state the obvious corollary for completeness. In the following statement $\rho_R (\mathcal B)$ denotes the spectral radius of the restriction of $\mathcal M$ to $R$, i.e.,~the spectral radius of the Bernoulli graphing on the random subspace. 

\begin{theorem}\label{thm:Bernoulli_spectral_radius_equality}
    Let $(G,o)$ be an ergodic unimodular random graph and $\cal B$ its associated Bernoulli graphing. Then $\rho (G,o) = \rho_R (\cal B)$.
\end{theorem}

As part of the proof of Theorem \ref{thm:intro.rel.ram}, we need to relate the quenched and annealed spectral radii of an ergodic unimodular random graph. For the former, return probabilities of the random walk are calculated on a random realization of the graph, while for the latter, one considers the expectation according to the randomness of the rooted graph. See~\ref{subsec:quenched_vs_annealed} for precise definitions. The authors believe the following result is worth mentioning in the introduction.

\begin{proposition}\label{thm:intro_quenched_vs_annealed} 
   For a unimdoular random graph $(G,o)$ the annealed spectral radius is the essential supremum of the quenched spectral radii. In particular, if $(G,o)$ is ergodic, the two notions coincide.
\end{proposition}

We found out during personal communication that Jan Grebik has also obtained Theorem~\ref{thm:intro.rel.ram} independently.



\subsection{Connection to new eigenvalues of finite random lifts}

We also show how Theorem~\ref{thm:intro.rel.ram} is a limiting version of results on finite random graphs. For finite $d$-regular graphs, the Alon-Boppana theorem states that, asymptotically, these cannot have larger spectral gaps than their universal cover, the $d$-regular infinite tree $T_d$. This gives a bound on how good expanders $d$-regular graphs can be. A finite $d$-regular graph $G$ is called \emph{Ramanujan} if it has such a large spectral gap, i.e.\ $\rho(G) \leq \rho(T_d)$. Constructing infinite families of $d$-regular graphs achieving this bound turned out to be a complicated problem, with examples coming from groups with Kazhdan's property (T) \cite{lubotzky1988ramanujan}. On the other hand, it was conjectured by Alon \cite{alon1986eigenvalues} and proved in a very technical 100-page paper 20 years later by Friedman \cite{friedman2008proof} that random $d$-regular graphs are \emph{almost Ramanujan} with high probability, see Theorem~\ref{thm:friedman}. Shorter, but still quite involved proofs have been found since, see \cite{bordenave2020new, chen2024new}.

One can think of the aforementioned result of Backhausz, Szegedy, and Virág on the Bernoulli graphing of a regular tree being Ramanujan (recalled as Theorem~\ref{thm:BSZV} below) as a measurable counterpart to Friedman's theorem. As is common in limiting settings, the measurable version is weaker, with a considerably softer proof. One can prove Theorem~\ref{thm:BSZV} directly from Friedman's result using the fact that random $d$-regular graphs locally converge to $T_d$, as sketched in section~\ref{subsubsec:ramanujan_graphs_and_graphings}. 

We highlight this connection by exhibiting another application. We exploit a strengthening of Friedman's theorem by Bordenave and Collins on new eigenvalues of random lifts of finite graphs \cite{bordenave2019eigenvalues} to prove a strengthening of Theorem~\ref{thm:intro.rel.ram} for a special class of unimodular random graphs. This class consists precisely of the unimodular random graphs arising as local limits of the models studied by Bordenave and Collins. 

\begin{theorem}\label{thm:qt.spec.eq}
    Let $G$ be a unimodular quasi-transitive quasi-tree, and $\cal B$ its Bernoulli graphing. Then $\sigma_R (\cal B) = \sigma (G)$.
\end{theorem}

To identify that the graphs arising as limits of the models treated by Bordenave and Collins are exactly unimodular quasi-transitive quasi-trees, we build on \cite{Stallings} and \cite{BK} to prove a characterization that we did not find explicitly stated in the literature.

\begin{theorem}\label{thm:char.intro}
    Let $G$ be a locally finite, connected graph. Then the following are equivalent.
    \begin{itemize}
         \item[(i)]  $G$ is a unimodular quasi-transitive quasi-tree. 
        \item[(ii)] There exists a free quasi-transitive action $\FF_d \actson G$ for some $d$.
        \item[(iii)]  There exists a finite graph $H$ and a map $\phi \colon \vec{E}(H) \to \FF_{d'}$ for some $d'$ with $\phi(v,u)=\phi(u,v)^{-1}$ for all $(u,v)\in \vec{E}(H)$ such that writing $\tilde H$ for the universal cover and  $\overline{\phi}:\pi_1(H) \to \FF_{d'}$ for the homomorphism induced by $\varphi$, we have $G\cong \tilde H / \ker (\overline\phi)$.
    \end{itemize}
\end{theorem}

Theorem~\ref{thm:char.intro} adds to the characterizations of (quasi-transitive or Cayley) graphs
that are quasi-isometric to trees by Antolín \cite{antolin2011cayley}, Hamman, Lehner, Miraftab and Rühmann \cite{Stallings}, Krön and Möller \cite{kron2008quasi}, Manning \cite{manning2005geometry} and Woess \cite{woess1989graphs}.


\vspace{0.3cm}
The paper is structured as follows. In Section~\ref{sec:prelim} we introduce the basic notions used throughout the paper, as well as the random and structured parts of the global spectrum and their basic properties. In Section~\ref{sec:warmup} we prove the results motivating our work, namely Theorems~\ref{thm:Alon-Boppana_intro} and \ref{thm:mikolay_intro}, as well as Theorem~\ref{thm:intro_quenched_vs_annealed} on the annealed versus quenched spectral radii. In Section~\ref{sec:ber.ram} we prove Theorem~\ref{thm:intro.rel.ram}. Finally, in Section~\ref{sec:fin.to.inf} we treat the connections to finite random graphs, proving Theorems~\ref{thm:qt.spec.eq} and \ref{thm:char.intro}.

\paragraph{Acknowledgements.}
    The authors are grateful to Miklós Abért, Charles Bordenave, Miko\l aj Fr{\k{a}}czyk, Ugo Giocanti, Samuel Mellick, Konrad Wr\'{o}bel, and Jan Grebík for helpful discussions. 

Héctor Jardón-Sánchez was supported by the ERC Starting Grant “Limits of Structures
in Algebra and Combinatorics” No.~805495 and the Dioscuri program initiated by the Max Planck Society,
jointly managed by the National Science Centre (Poland), and mutually funded by the Polish
Ministry of Science and Higher Education and the German Federal Ministry of Education
and Research. 

László Márton Tóth was supported by the National Research, Development and Innovation Fund grants KKP-139502 and STARTING 150723.

\section{Preliminaries}\label{sec:prelim}

\subsection{Fundamental notions and results}

In this section, we briefly introduce the basic notions of the paper: local convergence, unimodular random graphs, Bernoulli graphings, Markov operators, and Ramanujan graphs. Throughout this section, and for the rest of the paper, we assume that all graphs considered have degree bounded by $D\in \NN$.


\subsubsection{Unimodular random graphs} \label{subsubsec:URG}

For a detailed introduction to unimodular random graphs, graphings, and local convergence of finite graphs, we refer the reader to \cite[Part 4]{lovasz2012large}. This fundamental reference applies to Sections \ref{subsubsec:URG} through \ref{subsubsec:bernoulli_graphing}.

If $G$ is a graph, we denote by $V(G)$ and $E(G)$ its vertex and edge set respectively. If $\Xi$ is a standard Borel space, a \emph{$\Xi$-labelled graph} is a pair $(G,\xi)$ such that $G$ is a graph together with a map $l\colon V(G) \to \Xi$. A \emph{rooted graph} is a pair $(G,o)$ where $G$ is a graph and $o\in V(G)$ is a specified vertex called the \emph{root}. Similarly, a \emph{doubly rooted graph} is a triple $(G,o,o')$ such that $G$ is a graph and $o,o' \in V(G)$. Two rooted labelled graphs $(G,\xi,o)$ and $(H,\eta,p)$ are \emph{isomorphic}, denoted by $(G,\xi,o) \cong (H,\eta,p)$ if there exists a graph isomorphism $\alpha\colon G\to H$ such that $\alpha (o) = o'$ and $\xi = \eta\circ \alpha$. Isomorphism of doubly rooted graphs is defined analogously, requiring the graph isomorphism to preserve the roots. 

We use $\GG^\Xi$ (resp.~$\GGG^\Xi$) to denote the standard Borel space of isomorphism classes of connected rooted (resp.~doubly rooted) $\Xi$-labelled graphs with degree bounded by $D\in \NN$. 
In a slight abuse of notation we denote elements of $\GG^\Xi$ by pairs $(G,\xi,o)$, tacitly assuming that the latter is a representative in an equivalence class. We denote simply by $\GG$ or $\GGG$ the case in which $\Xi$ is a singleton, and in this case we omit reference to the labelling. We use $\GG^+$ to denote the setting where $\Xi = [0,1]$. (In this section, we will only work with unlabeled graphs, but will turn to labeled graphs in~\ref{subsubsec:bernoulli_graphing}.)

A \emph{random element} of $\GG^\Xi$ is a random variable $(G,\xi,o)\colon(\Omega,\PP)\to \GG^\Xi$, where $(\Omega,\PP)$ is an auxiliary standard probability space. A \emph{unimodular random graph} is a random element $(G,o)\in \GG$ satisfying the \emph{Mass Transport Principle}: for every Borel $f\colon \GGG \to \RR_{\geq 0}$, $$
\int_{\GG} \sum_{o' \in V(G)} f(G,o',o) d(G,o) = \int_{\GG} \sum_{o' \in V(G)} f(G,o,o') d(G,o).
$$

By an \emph{event} in $\GG$ we refer to a Borel subset of $\GG$. An event $U$ is \emph{rerooting invariant} if $(G,o) \in U$ implies $(G,o')\in U$ for every $o' \in V(G)$. A unimodular random graph $(G,o)$ is \emph{ergodic} if for every rerooting invariant event $U \subset \GG$, $$
\PP [ (G,o) \in U] = 0 \textrm{ or } 1.
$$

\subsubsection{Local convergence}

A sequence $G_n$ of bounded degree finite graphs is said to be \emph{locally (or Benjamini-Schramm) convergent} if, when choosing a uniform random root $o_n \in V(G_n)$, the random rooted graph $(G_n, o_n)$ converges in distribution in $\GG$. For further reference, we denote by $\nu_{G_n}$ the distribution of $(G_n, o_n)$. The limit of a convergent sequence is not a finite graph anymore. In fact, there are two established ways of representing limits of convergent sequences. One is \emph{unimodular random graphs}, defined in the previous section (\ref{subsubsec:URG}). The other is \emph{graphings}, which we proceed to introduce (\ref{subsubsec:graphing}). The core topic of our investigation is the interplay between the two viewpoints.

\subsubsection{Graphings} \label{subsubsec:graphing}

If $X$ is a standard Borel space, a \emph{Borel graph} is a symmetric Borel subset $\cal G \subset X \times X$. The degree map $\deg \colon X \to \NN$ is Borel and assigns to each $x\in X$ its $\cal G$ degree $\deg(x)$. We will only consider Borel graphs for which there exists $D\in \NN$ such that $\deg (x) \leq D$ for every $x\in X$.

A \emph{graphing} is a triple $(\cal G,X,\mu)$ where $\cal G$ is a Borel graph on a standard Borel space $X$ and $\mu$ a Borel probability measure on $X$ such that the following Mass Transport Principle is satisfied: for every $f\colon \cal G \to \RR_{\geq 0}$, $$
\int_X \sum_{y\in V(G_x)} f(y,x) d\mu(x) = \int_X \sum_{y\in V(G_x)} f(x,y) d\mu(x).
$$
Let us denote the $\cal G$-connected component of a vertex $x\in X$ by $\cal G_x$. Picking $x$ randomly according to $\mu$, the random rooted graph $(\cal G_x,x)$ turns out to be unimodular, and we denote its distribution by$\nu_{\mathcal{G}}$. 
We say $\mathcal{G}$ \emph{represents} a unimodular random graph with distribution $\nu$ if $\nu_{\mathcal{G}}=\nu$, and that a sequence of finite graphs $(G_n)$ \emph{locally converges} to $\cal G$ if $\nu_{G_n} \to \nu_{\cal G}$ weakly.

Every unimodular random graph can be represented by some graphing, for example, the Bernoulli graphing defined below. This representation is very much not unique: graphings that are not measurably isomorphic can represent the same unimodular random graphs, and hence be limits of the same sequence of finite graphs.

\subsubsection{Bernoulli graphings} \label{subsubsec:bernoulli_graphing}

The \emph{Bernoulli graph} $\cal B$ is the Borel graph on $\GG^+$ defined by setting an edge between $(G,\xi,o)$ and $(H,\eta,p)$ if and only if there exists a neighbour $o'$ of $o$ in $G$ such that $(G,\xi,o') \cong (H,\eta,p)$. If $(G,o)$ is a unimodular random graph, its \emph{Bernoulli graphing} is $(\cal B,\GG^+,\mu)$ where the measure $\mu$ is defined by $$
\int_{\GG^+} f d\mu = \int_{\GG} \int_{[0,1]^{V(G)}} f(G,\xi,o) d\Leb^{\otimes V(G)} (\theta) d(G,o).
$$
whenever $f\colon \GG^+ \to \RR_{\geq 0}$ is any Borel map and $\Leb$ denotes the Lebesgue measure on $[0,1]$. In order to simplify notation, in the sequel we will denote $\beta_{G} := \Leb^{\otimes V(G)}$. 

One can sample a $\mu$-random element from $\GG^+$ by sampling the unimdoular random graph $(G,o)$ first, and then placing iid uniform $[0,1]$-valued labels on its vertices. Observe that, despite the vertex set of the Bernoulli graphing not depending on $(G,o)$, the support of $\mu$ does depend on the distribution of $(G,o)$.

\begin{remark}\label{rem:change.labels}
    By the Isomorphism Theorem for standard probability spaces \cite[Theorem 17.41]{kech}, we may assume to our convenience that the labels of the Bernoulli graphing are sampled in an iid fashion from any other standard probability space $(\Xi, \lambda)$ instead of $([0,1],\Leb)$. This observation will be used in several proofs throughout the paper. 
\end{remark}

\subsubsection{Operators and spectra}

For the next two subsections, we turn our attention to operator theory. We now recall a few fundamental facts about the spectral theory of operators on Hilbert spaces. Then we focus our attention on Markov operators.

 Let $(H, \langle,\rangle)$ be a Hilbert space and $T$ a bounded operator on $H$. The \emph{spectrum} of $T$ is the closed subset $$
\sigma(T) := \{\lambda \in \mathbb{C} : \lambda - T\ \text{is not invertible}\},
$$
and the \emph{spectral radius} of $T$ is the quantity $$
\rho(T) := \sup_{\lambda\in \sigma(T)} |\lambda|.
$$

Recall that 
self-adjoint operators $T$ have real spectrum $\sigma(T)\subset \RR$, and every element in $\sigma(T)$ is an approximate eigenvalue: 

\begin{proposition}[Prop.~6.9 in \cite{Conway}]\label{prop:app.eval}
    Let $T$ be a self-adjoint operator on a Hilbert space $H$. Then $\lambda\in \sigma(T)$ if and only if there exists a sequence $(v_n)_n \subset H$ such that $\|v_n\| = 1$ for every $n\in \NN$ and $\|(\lambda - T)v_n\| \to 0$.
\end{proposition}

The following version of Beurling's spectral radius formula \cite[Theorem 1.2.7]{murphy}, which can be found in \cite{backhausz2015ramanujan}, will be useful throughout the paper. 

\begin{proposition}
    Let $T$ be a bounded self-adjoint operator on a Hilbert space and $U \subset H$ a set of normalised vectors spanning $H$. Then $$
    \rho(T) = \sup_{v\in U} \left(\limsup_{n\to \infty} \sqrt[n]{\langle T^n v, v \rangle} \right).
    $$ 
\end{proposition}

\subsubsection{Markov operators} \label{subsec:markov_operators}

Let $G$ be a connected bounded degree graph. In a slight abuse of notation, we denote by $\ell^2 (G, \deg)$ the Hilbert space of maps $f \colon V(G) \to \mathbb{C}$ with inner product  $$
\langle f, g\rangle = \sum_{u\in V(G)} f(u) \overline{g(u)} \deg (u).
$$
The \emph{Markov operator of $G$} is the self-adjoint operator on $\ell^2 (G,\deg)$ defined by $$
(M_G f) (u) = \frac{1}{\deg (u)} \sum_{(u,v) \in E(G)} f(v),
$$
for every $f \in \ell^2 (G,\deg)$.

Beurling's spectral radius formula establishes a connection between the spectral theory of Markov operators of graphs and the return probabilities of the simple random walk on them.

\begin{theorem}[e.g.\ Prop.\ 6.6.\ in \cite{LP}]
    Let $G$ be a connected bounded degree graph and let $X_n$ denote the position of the simple random walk after $n$ steps on $G$ with starting vertex $o \in V(G)$. Then $$
    \rho(G) := \limsup_{n\to \infty} \sqrt[n]{\mathbb{P} [X_{n} = o]}
    $$
\end{theorem}

If $(G,o)$ is a unimodular random graph, then the spectral radius $\rho(G,o) := \rho(M_G)$ becomes a real-valued rerooting invariant random variable. Thus, when $(G,o)$ is ergodic $\rho (G,o)$ becomes essentially a constant, which we call the \emph{spectral radius of} $(G,o)$.

Let $(\cal G,X,\mu)$ be a graphing. We let $\mu^\ast$ denote the \emph{degree-biased version} of $\mu$, defined by $$
\int_X f d\mu^\ast = \frac{\int_{X} f \deg d\mu}{\int_{X} \deg d\mu},
$$
for every Borel map $f\colon X \to \RR_{\geq 0}$.
 
In this case we donete the Markov operator by calligraphic $\cal M$, and note that it is again a self-adjoint operator, this time on $L^2 (X,\mu^\ast)$.
The \emph{spectral radius} of $\cal G$ is $\rho (\mathcal G) := \rho (\cal M|_{L_0^2 (X,\mu^\ast)})$, where $L_0^2 (X,\mu^\ast)$ is the closed subspace of $L^2 (X,\mu^\ast)$ consisting of random variables orthogonal to constants, and $\cal M|_{L_0^2 (X,\mu^\ast)}$ is the restriction of $\mathcal M$ to this $\cal M$-invariant subspace.

\subsubsection{Ramanujan graphs and graphings}\label{subsubsec:ramanujan_graphs_and_graphings}

Here we recall the theorems on Ramanujan graphs mentioned in the introduction.

\begin{theorem}[Alon-Boppana] \label{thm:alon_boppana}
   Let $(G_n)$ be a sequence of $d$-regular graphs with $|V(G_n)| \to \infty$. Then
  \[\liminf \rho(G_n) \geq \rho(T_d) = \frac{2\sqrt{d-1}}{d}.\]
\end{theorem}

A finite graph $G$ is said to be Ramanujan, if it satisfies the reverse inequality, that is, $\rho(G) \leq \rho(T_d)$. By a deep theorem of Friedman, random $d$-regular graphs are almost Ramanujan, they satisfy the reverse inequality with arbitrarily small error.

 \begin{theorem}[Friedman \cite{friedman2008proof},\cite{bordenave2020new}] \label{thm:friedman}
    Let $G_n$ be a $d$-regular graph on $n$ vertices, chosen uniformly randomly (assuming that $nd$ is even). Then
\[\forall \varepsilon > 0, \quad \mathbb{P}\left[ \rho(G_n)\leq \frac{2\sqrt{d-1}}{d} + \varepsilon \right] \to 1, \ \mathrm{\it as } \  n \to \infty.\]
\end{theorem}

In the measurable setting the analog of Friedman's theorem is the following.

\begin{theorem}[Backhusz, Szegedy, Virág \cite{backhausz2015ramanujan}] \label{thm:BSZV}
    The Bernoulli graphing $\mathcal{B}$ of the $d$-regular tree $T_d$ is Ramanujan, meaning $\rho(\mathcal{B}) = \rho (T_d)$.
\end{theorem}

As a motivation for Section~\ref{sec:fin.to.inf} where we prove Theorem~\ref{thm:qt.spec.eq}, we now sketch an argument deriving Theorem~\ref{thm:BSZV} from Theorem~\ref{thm:friedman}.

\begin{proof}[Sketch of proof]
    Assume towards contradiction that $\rho(\mathcal{B})> \rho (T_d)$. This is witnessed by a unit vector $f  \in L_0^2(V(\mathcal{B}))$, that is, $\langle \mathcal{M}f,f \rangle \geq \rho(T_d)+\varepsilon$. We can assume without loss of generality that $f$ depends only on the random labels within the $r$-ball around the root of $T_d$. 

    Generate a sequence $G_n$ of random $d$-regular graphs. We can mimic $f$ on the $G_n$ by placing iid uniform random labels on the vertices. Denote these labelings $l_n:V(G_n) \to [0,1]$. As the $G_n$ Benjamini-Schramm converge to $T_d$, for $n$ large from most vertices the $r$-ball looks like $B_r(T_d)$. We can thus define a function $f_n:V(G_n) \to \mathbb{R}$ by $f_n(v)= f\big(B_r(G_n,v), v, l_n\big)$. With probability 1 the $f_n$ turn out to be asymptotic witnesses for $\rho(G_n) \geq \rho(T_d) + \varepsilon$, contradicting Theorem~\ref{thm:friedman}. 
\end{proof}

\subsection{The random and structured subspaces} \label{subsec:random_and_structured}

Let $(G,o)$ be a unimodular random graph with law $\nu$ and $(\cal B,\GG,\mu)$ denote its associated Bernoulli graphing. The \emph{sampling map} $s\colon (G,\xi,o) \in \GG^+ \to (G,o)\in \GG$ is such that $s_\ast \mu^\ast = \nu^\ast$ and thus establishes an isometric embedding $\iota  \colon L^2 (\GG,\nu^\ast)\hookrightarrow L^2 (\GG^+,\mu^\ast)$ by letting $\iota f  := f \circ s$ for every $f \in L^2 (\GG, \nu^\ast)$.

\begin{definition}
    The \emph{structured subspace} of $L^2 (\GG^+,\mu^\ast)$ is $S := \ran (\iota)$. The \emph{random subspace} of $L^2 (\GG^+,\mu^\ast)$ is the orthogonal complement $R := S^\perp$.
\end{definition}

We now establish the basic properties of these subspaces, and characterize functions in $R$ as those whose conditional expectation is $0$ given almost every realization of $(G,o)$.

\begin{proposition}\label{prop:proj}
    Let $P$ be the orthogonal projection on $S$. Then, for every $f\in L^2 (\GG^+,\mu^\ast)$,  $$
    (Pf)(G,\xi,o) = \int_{[0,1]^{V(G)}} f(G,\eta,o) d \beta_G (\eta),
    $$
    for $\mu^\ast$-almost every $(G,\xi,o)\in \GG$.
\end{proposition}

\begin{proof}
The map $\iota$ is a partial isometry, so  $\iota\iota^\ast$ is the orthogonal projection on $S$ by \cite[Theorem 2.3.3]{murphy}. For any $g\in L^2 (\GG,\nu^\ast)$ and $f\in L^2 (\GG^+,\mu^\ast)$, we have \begin{align*}
         \langle  f , \iota g \rangle &= \int_{\GG^+} \overline{g(G,o)} f(G,\eta,o) d\mu^\ast (G,\eta,o) \\
         &= \int_{\GG} \overline{g(G,o)} \int_{[0,1]^{V(G)}} f(G,\eta,o) d\beta(\eta) d\nu^\ast(G,o).
    \end{align*}
Therefore, the adjoint $\iota^\ast$ is defined for any $f\in L^2 (\GG^+,\mu^\ast)$ by \begin{equation}\label{eq:iota.star}
     \iota^\ast f (G,o) = \int_{[0,1]^{V(G)}} f(G,\eta,o) d\beta(\eta).
\end{equation}
Then, the orthogonal projection $P$ on $S$ is determined for any $f\in L^2 (\GG^+,\mu^\ast)$ by $$
Pf (G,\xi,o) = \iota\iota^\ast f (G,\xi,o) = \int_{[0,1]^{V(G)}} f(G,\eta,o) d\beta_G (\eta).
$$

\end{proof}

As a corollary of the above proposition, we get the following characterization of the random subspace.

\begin{corollary}\label{cor:random.char}
    Let $f\in L^2 ( \GG^+, \mu^\ast)$. Then $f\in R$ if and only if $$
    \int_{[0,1]^{V(G)}} f(G,\xi,o) d\beta_G (\xi) = 0,
    $$
    for $\nu$-almost every $(G,o)$.
\end{corollary}

Recall that a \emph{block factor of radius} $r\in \NN$ is an element $f \in L^2 (\GG^+,\mu^\ast)$ such that for any two $(G,\xi,o),(H,\eta,p)\in \GG^+$, $$
(B_G (o,r) , \xi|_{B_G (o,r)},o) \cong(B_H (p,r) , \xi|_{B_H (p,r)},p) \Rightarrow f (G,\xi,o) = f(H,\eta,p)
$$
I.e., a block factor is a random variable on $\GG^+$ whose values are determined only by the labels on a ball of finite radius around the root. 

\begin{corollary}\label{cor:proj.block}
    If $f\in L^2 (\GG^+,\mu^\ast)$ is a block factor, so are $Pf$ and $(1-P)f$
\end{corollary}
    \begin{proof}
      By Proposition \ref{prop:proj}, the random variable $Pf$ does not depend on the labels.
    \end{proof}

We conclude this section observing that the decomposition of $L^2 (\GG^+,\mu^\ast)$ into the random and structured subspaces is invariant under the Markov operator.

    \begin{corollary}\label{cor:proj.markov}
 Let $\cal M$ denote the Markov operator on $L^2 (\GG^+,\mu^\ast)$. Then $P\cal M = \cal M P$.
    \end{corollary}

    \begin{proof}
        We argue by direct computation on a random variable $f\in L^2 (\GG^+,\mu^\ast)$. For almost every $(G,\xi,o) \in \GG^+$,
        \begin{align*}
            P\cal M f (G,\xi,o) &= \int_{[0,1]^{V(G)}}\frac{1}{\deg(o)} \sum_{(o',o)\in E(G)} f (G,\eta,o') d\beta_G(\eta)\\
            &= \frac{1}{\deg(o)} \sum_{(o',o)\in E(G)} \int_{[0,1]^{V(G)}} f(G,\eta,o')d\beta_G(\eta)\\
             &= \cal M P f (G,\xi,o).
        \end{align*}
    \end{proof}

\subsection{The structured spectrum} \label{subsec:structured_spectrum}

After Corollary \ref{cor:proj.markov}, we obtain a decomposition of the Bernoulli spectrum $\sigma (\mathcal B) := \sigma (\mathcal M)$ as a union $\sigma (\mathcal M|_S) \cup \sigma (\mathcal M|_R)$. We refer to $\sigma_S (\mathcal B) := \sigma (\mathcal M|_S)$ as the \emph{structured spectrum} and to $\sigma_R (\mathcal B) := \sigma (\mathcal M|_R)$ as the \emph{random spectrum}. In this section, we show that the structured spectrum is determined by a Markov chain on $\GG$.

\begin{definition}
    The \emph{skeleton Markov chain} $\cal S$ is defined on $\GG$ by $$
    p_{\cal S}((G,u),(H,v)) = \frac{|\{w\in N_G (u) : (G,w) \cong (H,v)\}|}{\deg_G (u)}.
    $$ 
\end{definition}

Degree-biased versions of unimodular random graph laws are stationary measures for the skeleton.

\begin{proposition}
    Let $(G,o)$ be a unimodular random graph with law $\nu$. Then $\nu^\ast$ is a stationary measure for $\cal S$.
\end{proposition}

\begin{proof}
    For any event $U\subset \GG$, by the Mass Transport Principle
    \begin{align*}
    &\int_{\GG} p((G,o), U) d\nu^\ast (G,o)\\
    &=\frac{1}{\EE_\nu [\deg_G (o)]}\int_{\GG} \deg_G (o,U) d\nu (G,o)\\
    &= \frac{1}{\EE_\nu [\deg_G (o)]}\int_{U} \deg_G (o) d\nu (G,o)\\
    &= \nu^\ast (U).
    \end{align*}
\end{proof}

In the next proposition $\cal N$ denotes the Markov operator of $\cal S$ on $L^2 (\GG,\nu^\ast)$.

\begin{proposition}
For a unimodular random graph $(G,o)$ we have $\mathcal N = \iota^\ast \mathcal M \iota$. In particular $\sigma_S (\mathcal B) = \sigma (\mathcal N)$. 
\end{proposition}

\begin{proof}
 By Equation (\ref{eq:iota.star}), we have that for an arbitrary $f\in L^2 (\GG,\nu^\ast)$, $$
    \iota^\ast \cal M \iota  f (G,o)= \int_{[0,1]^{V(G)}} \frac{1}{\deg(o)} \sum_{(o',o)\in E(G)} f(G,o') d\beta_G(\eta)  = \cal N f (G,o).
    $$
\end{proof}

\section{Initial results}\label{sec:warmup}

In this section, we prove the results that serve as the entry points to the investigation of the relationship between the local and global spectrum.
\begin{enumerate}[(a)]
    \item In \ref{subsec:random_contains_local} we establish that the random (and therefore the global) spectrum of the Bernoulli contains the local. As a corollary, we get containment of the local spectrum for all aperiodic graphings. 
    \item In \ref{subsec:Mikolay_example} we present an example, due to Fr{\k{a}}czyk, of a Bernoulli graphing where $\rho(G,o)<1$, but $\rho(\cal B)=1$.
    \item In \ref{subsec:quenched_vs_annealed} we show that the annealed spectral radius of $(G,o)$ is the essential supremum of the quenched one, so the two coincide for ergodic URG's.
\end{enumerate}   

\subsection{The random spectrum contains the local spectrum} \label{subsec:random_contains_local}

\begin{theorem}\label{thm:spec.containment}
    Let $(G,o)$ be an ergodic unimodular random graph, and $\mathcal{B}$ its Bernoulli graphing. Then $\sigma (G,o) \subset \sigma_R (\cal B)$ almost surely. 
\end{theorem}

\begin{proof}

Let $\lambda\in \sigma(G,o)$. By ergodicity we have that $\lambda \in \sigma(M_{(G,o)})$ for almost every $(G,o) \in \GG$. The operator $M_{(G,o)}$ is self-adjoint, by Proposition \ref{prop:app.eval} we have that $\lambda$ is an approximate eigenvalue. It follows that there exists a family of measurable assignments $(G,o)\in \GG \mapsto f_n^{(G,o)} \in \ell^2 (G)$ such that \begin{itemize}
    \item[(a)] $\|f_n^{(G,o)}\| = 1$ for every $n$,
    \item[(b)] the $f_n^{(G,o)}$ are finitely supported, and
    \item[(c)] $\|(M_G - \lambda) f_n^{(G,o)}\|^2 \leq \frac1n$. 
\end{itemize}
By the Isomorphism Theorem for standard probability spaces \cite[Theorem 17.41]{kech}, we may assume that Bernoulli labels on vertices are of the form $\xi = (\xi_1,\xi_2)$, where $\xi_1$ is a uniform iid label on $[0,1]$, and $\xi_2$ is a uniformly sampled iid label from $\{-1,1\}$.

Let $A_n \subset \GG^+$ be the Borel subset of those $(G,\xi,o)\in \GG^+$ such that for every other $(G,\xi,o')$ with $$
B_1(\supp (f_n^{(G,o)}) )\cap B_1 (\supp (f_n^{(G,o')})) = \emptyset,
$$
either $\diam (\supp(f_n^{(G,o')})) > \diam(\supp(f_n^{(G,o)}))$ or $\xi_1(o) > \xi_1(o')$. By (b) we have $\mu (A_n) > 0$. Let $B_n \subset \GG^+$ be the Borel subset consisting of those $(G,\xi,o) \in \GG^+$ such that $o \in \supp(f_n^{(G,o')})$ for some $(G,\xi,o')\in A_n$. By construction,  such $o'$ is uniquely determined for almost every $(G,\xi,o)$ independently of $\xi_2$, so we denote it as $\phi_n (G,\xi_1,o)$. 

    Our aim is to produce a sequence of functions $F_n \in R$ such that $\|(\cal M - \lambda) F_n\| \to 0$ and $\|F_n\| = 1$. We define them as follows. For each $n$, let $F_n \in L^2 (\GG^+,\mu^\ast)$ be defined by $$
    F_n (G,\xi,o) := \frac{1}{\mu(A_n)^{\frac12}\deg_G(o)^{\frac12}} f_n^{(G,\phi_n (G,\xi_1,o))} (o)\xi_2 (\phi_n (G,\xi_1,o)) ,\ \text{if}\ (G,\xi,o)\in A_n,
    $$ 
    and 0 otherwise. On the one hand, by the Mass Transport Principle and (a),$$
    \|F_n\|^2 = \int_{B_n} \frac{|f_n^{\phi_n(G,o,\xi_1)} (o)|^2}{\mu(A_n)} d\mu(G,\xi,o) = \frac{1}{\mu(A_n)} \int_{A_n} d\mu(G,\xi,o) = 1.
    $$
    On the other hand, again by the Mass Transport Principle, but now by (c), \begin{align*}
        \|(\cal M - \lambda) F_n \|^2 &= \frac{1}{\mu(A_n)}\int_{A_n} \|(M_G -\lambda)f_n^{(G,o)}\|^2 d\mu(G,o) \leq \frac{1}{n} \to 0.
    \end{align*}

The above shows that $\lambda \in \sigma(\cal M)$. To conclude the proof, it is only left to show that $F_n \in R$. To do this, by Corollary \ref{cor:random.char}, it suffices to show that for almost every $(G,o)$,$$
\int_{[0,1]^{V(G)}}\int_{\{-1,1\})^{V(G)}} f_n^{(G,\phi_n (G,\xi_1,o))} (o)\xi_2 (\phi_n (G,\xi_1,o)) d\beta_G(\xi_1) d\beta_G(\xi_2) = 0,
$$
However, since $\phi_n$ does not depend on $\xi_2$, the above integral equals  $$
\int_{[0,1]^{V(G)}} f_n^{(G,\phi_n (G,\xi_1,o))} (o) \int_{\{-1,1\})^{V(G)}}\xi_2 (\phi_n (G,\xi_1,o))d\beta_G(\xi_2) d\beta_G(\xi_1),
$$
where the term $\int_{\{-1,1\})^{V(G)}}\xi_2 (\phi_n (G,\xi_1,o))d\beta_G(\xi_2)=0$ for almost every $\xi_1$, concluding the proof.
\end{proof}

A corollary of the above Theorem is the following. 

\begin{corollary}\label{cor:alon.boppana}
   Let $(\cal G,X,\mu)$ be an aperiodic ergodic graphing with associated unimodular random graph $(G,o)$. Then $\sigma (G,o) \subset \sigma (\cal G)$.   
\end{corollary}

\begin{proof}
    By minimality of aperiodic Bernoulli  graphings \cite[Corollary 7.7]{HLS}, the Bernoulli graphing $\cal B$ of $(G,o)$ is locally-global contained in $\cal G$. If thus follows that $\sigma (\cal B) \subset \sigma (\cal G)$, and the conclusion follows from Theorem \ref{thm:spec.containment}.
\end{proof}

\begin{remark}
    In any graphing $(\cal G,X,\mu)$ a sampling map $s\colon X \mapsto \GG$ can be defined by $x\in X \mapsto (\mathcal G_x, x)\in \GG$, where $\mathcal G_x$ denotes the $\mathcal G$-connected component of $x$. Structured and random subspaces of $L^2 (\mathcal G,X,\mu)$ can thus be defined and the decomposition of the spectrum $\sigma (\mathcal G)$ into a structured and random part carries over. 

    The content of Section \ref{subsec:structured_spectrum} holds in the generality of graphings with the above sampling map, and the proof of Theorem \ref{thm:spec.containment} in this generality shows that $\sigma (G,o) \subset \sigma_R (\mathcal G^+)$, where $\sigma_R (\mathcal G^+)$ is the random part of the spectrum of the Bernoulli extension $\mathcal G^+$. 

    It is important to remark that Theorem \ref{thm:spec.containment} does not hold for arbitrary graphings, as there are graphings for which the sampling map is an isomorphism up to null sets, and so the random subspace is trivial. For example, if $(T,o)$ is a unimodular Galton-Watson tree with law $\nu$, and let $\mathcal G$ be defined on $\GG$ by $$
    (T,o) \mathcal G (T,o') \ \text{if and only if}\ (o,o') \in T.
    $$
    Then $(\mathcal G, \GG,\nu)$ is a graphing by \cite[Lemma 18.40]{lovasz2012large} and, upon dropping a $\nu$-null set, the sampling map coincides with the identity.
\end{remark}

\subsection{A URG with spectral gap but non-expanding Bernoulli graphing} \label{subsec:Mikolay_example}

In this section we construct an ergodic unimodular random graph $(G,o)$ with spectral radius $\rho (G,o) < 1$ but such that $\rho (\mathcal B) = 1$, where $\mathcal B$ denotes the Bernoulli graphing of $(G,o)$. As we discuss below in Remark \ref{rem:example} this example is a motivation for Theorem \ref{thm:intro.rel.ram}. Such an example was first communicated to the authors by Miko{\l}aj Fr{\k{a}}czyk. 

\begin{proof}[Proof of Theorem~\ref{thm:mikolay_intro}]
Let $\FF_2 = \langle a,b  \rangle$ denote the free group on two generators and $\phi \colon \FF_2 \to \ZZ$ the homomorphism defined by $\phi(a)=\phi(b) = 1$. We construct a unimodular random coloured graph $(T,c,o)$ by the following sampling rule. The graph $T$ is always the 4-regular tree, with vertices identified with elements of $\FF_2$ by the standard Cayley graph of the latter, rooted at the identity $o$. The colouring $c = (i,b)$ consists of two labels. The label $i$ is determined by choosing $i_o \in \{0,1,2\}$ uniformly at random, and letting $i(g) = \phi(g) + i_o \ (\textrm{mod }  3)$. The label $b$ is obtained by pulling back a Bernoulli $\{R,B\}$-labelling of $\ZZ$
to $T$ by $\phi$.

We now construct $(G,o)$ as a factor of $(T,c,o)$. First, fix a bijection $l\colon \{0,1,2\}\times \{R,B\} \to [6]$. Then, sample from $(T,c,o')$, and let $G$ be the tree constructed by attaching a path $P_v$ of length $l(v) - 1$ to each $v\in V(T)$. The root $o$ is placed on a vertex of $P_{o'} \cup\{o'\}$ uniformly at random.

We claim that $(G,o)$ satisfies the conclusions of the proposition. Ergodicity of $(G,o)$ follows from the fact that the diagonal action of $\ZZ$ on the product of the $\{R,B\}$-Bernoulli shift and $(\ZZ / 3\ZZ)$ is ergodic, as $(T,c,o)$ is a factor of this action. Furthermore, $\rho(G,o) <1$ follows from Kesten's Theorem \cite{kesten1959full,kesten1959symmetric}, as the Cheeger constant $h$ of $G$ is easily seen to satisfy the inequality $$
h(G) \geq \frac{h(T)}{6} > 0.
$$

Finally, to show that $\rho(\cal B) = 1$ we prove that the associated Bernoulli graphing $\mathcal B$ is not strongly ergodic. To do this, let $(G,o,c)$ be a Bernoulli random coloured version of $(G,o)$. By construction, the graph $(G,o)$ factors onto $(T,o,i,b,c)$, as the labels $i$, $b$, and $c$ can be read from the lengths of the paths attached to each vertex of $T$. The labels $i$ allow us to identify the vertices of $T$ with $\FF_2$ and thus to reconstruct the homomorphism $\phi$ in such a way that $\phi(g) = \phi(h)$ implies $b(g) = b(h)$. We can thus map $b$ down by $\phi$ to obtain a Bernoulli random colouring of $\ZZ$ as a factor. The latter is not strongly ergodic, so $\mathcal G$ cannot be strongly ergodic either.
\end{proof}

\begin{remark}\label{rem:example}
    Note that  Bernoulli labeling plays no part in the definition of the factor onto $2^\mathbb{Z}$, it can be built solely from the graph structure of $(G,o)$. This observation is the motivation for our first main result, Theorem~\ref{thm:intro.rel.ram}, where we formally prove that this is the only possible reason for $\rho(G,o) < \rho(\cal B)$.
\end{remark}

\begin{remark}
    The same construction as above, with slight modifications, can be carried out using any non-amenable finitely generated indicable group. The original idea for the above construction, to which the latter can actually be generalized, starts from a non-amenable, non-property (T) group $\Gamma$. On the one hand, the spectral radius of the group by non-amenability is strictly less than 1. On the other hand, by the Connes-Weiss Theorem, the group has an ergodic not strongly ergodic action. By the Jones-Schmidt Theorem \cite{JS}, there exists a $\Xi$-marking of $\Gamma$ with a non-trivial factor onto $2^\ZZ$.
\end{remark}

\begin{remark}
 In section ~\ref{subsubsec:ramanujan_graphs_and_graphings} we sketched an argument showing that Friedman's theorem implies that Bernoulli graphings of regular trees are Ramanujan. Here, the same argument shows that these examples, even though they are non-amenable in a strong sense, they are not local limits of expander sequences.  
\end{remark}

\subsection{Quenched versus annealed spectral radius}
\label{subsec:quenched_vs_annealed}

In this epigraph we show that for ergodic unimodular random graphs the spectral radius may be computed in both quenched and annealed forms. This technical observation is needed to prove Theorem \ref{thm:intro.rel.ram}.

 For each $(G,o) \in \GG$ we denote $$
p_{2n} (G,o) := \PP [X_{2n} = o],
$$ 
where $X_{i}$ is the $i$-th step of the uniform random walk on $G$ starting at the root $o$. The \emph{quenched spectral radius} is the re-rooting invariant random variable $$
\rho_q (G,o) = \lim_n \sqrt[2n] {p_{2n} (G,o)}.
$$
If $(G,o)$ is a unimodular random graph, then its \emph{annealed spectral radius} is defined as $$
\rho_a  = \lim_n \sqrt[2n]{\EE [p_{2n}]}.
$$

The aim of this section is to prove the following equality.

\begin{theorem}\label{thm:q.vs.a}
    Let $(G,o)$ be a unimodular random graph. Then $
    \rho_a = \esup_{(G,o)\in \GG} \rho_q (G,o).
    $ In particular, if $(G,o)$ is ergodic, then $\rho_a  = \rho_q (G,o)$ almost surely.
\end{theorem}

\begin{proof}
    Fix any element $(G,o) \in \GG$. A standard computation shows that $p_{2nl} (G,o) \geq p_{2n}(G,o)^l$ for any $n,l\in \NN$. It thus follows that $$
    \rho_q (G,o)^{2n} = \lim_{l\to\infty} \sqrt[l]{p_{2nl} (G,o)} \geq p_{2n} (G,o).
    $$
    In particular, we have the bound $$
    \esup \rho_q (G,o)^{2n} \geq p_{2n} (G,o),
    $$
    almost surely. It follows that $$
    \esup \rho_q (G,o)^{2n} \geq \EE [p_{2n} (G,o)],
    $$
    and therefore $$
    \esup \rho_q (G,o) \geq \rho_a.
    $$

    Suppose now that $\rho_a < \esup \rho_q (G,o)$. Then, there exist $\eps,\delta > 0$ such that, $$
    \PP[\rho_q (G,o) \geq \rho_a + 2\delta] \geq 2\eps.
    $$ 
    By definition of $\rho_q$, the above implies that for some $n_0 \in \NN$, for every $n\geq n_0$ we have that $$
    \PP[p_{2n} (G,o) \geq (\rho_a + \delta)^{2n}] \geq \eps.
    $$
    By definition of the annealed spectral radius we have that $$
    \rho_a \geq \sqrt[2n]{\eps (\rho_a + \delta)^{2n}} = \sqrt[2n]{\eps} (\rho_a + \delta) \to \rho_a +\delta.
    $$
    The latter implies that $\rho_a > \rho_a$, what is a contradiction. The conclusion follows.
\end{proof}

\section{Bernoulli graphings are relatively Ramanujan}\label{sec:ber.ram}

In this section we prove Theorem~\ref{thm:intro.rel.ram}, relying on the preparation done  in~\ref{subsec:random_and_structured} and \ref{subsec:quenched_vs_annealed}. 

\begin{proof}[Proof of Theorem~\ref{thm:intro.rel.ram}]

We aim to  show that $\rho_R (\cal B) \leq \rho(G,o)$. Since projections of block factors to $R$ are block factors themselves by Corollary~\ref{cor:proj.block}, it follows that block factors are dense in $R$. Let $f\in R$ be a normalised block factor of radius $r$. By \cite[Lemma 2.4]{backhausz2015ramanujan}, to conclude the proof it suffices to show that \begin{equation}\label{eq:goal.ramanujan}
    \sqrt[n]{\langle \mathcal M^n f,f \rangle} \leq (1+ o(1)) \rho (G,o).
\end{equation} 

Let us begin the computations. By definition, $$
\langle\cal M^n f , f \rangle = \int_{\GG^+} \sum_{v\in V(G)} p_n (o,v) f(G,\xi,v) \overline{f(G,\xi,o)} d\mu^\ast (G,\xi,o).
$$    
We may separate the above integral into two summands \begin{align*}
    \int_{\GG}\sum_{v\in B_{2r} (o)} p_n (o,v)\int_{[0,1]^{V(G)}}  f(G,\xi,v) \overline{f(G,\xi,o) }d\beta(\omega) \deg_G (o)d(G,o)\\
    + \int_{\GG} \sum_{v\not\in B_{2r} (o)} p_n (o,v)\int_{[0,1]^{V(G)}}   f(G,\xi,v) \overline{f(G,\xi,o)}d\beta(\omega) \deg_G (o)d(G,o).
\end{align*}
By the block factor assumption and Corollary \ref{cor:random.char}, we have that, whenever $v\not\in B_{2r} (o)$, the random variables $f(G,\xi,v)$ and $f(G,\xi,o)$ are independent. Therefore, \begin{align*}
&\int_{[0,1]^{V(G)}}  f(G,\xi,v) \overline{f(G,\xi,o)} d\beta (\omega) \\
&= \int_{[0,1]^{V(G)}}   f(G,\xi,v)  d\beta (\omega)\int_{[0,1]^{V(G)}}   \overline{f(G,\xi,o)} d\beta (\omega) = 0.
\end{align*} 
Hence, there is only a contribution from the first term of the sum. 

Using Cauchy-Schwarz this second term may be bounded above by \begin{equation}\label{eq:sec.6.eq.1}
    \int_{\GG}  \sum_{v\in B_{2r} (o)} p_n (o,v) \|f(G,\cdot,v)\|_2 \|f(G,\cdot,o)\|_2  \deg_G(o)d (G,o).
\end{equation}
If we let $\GG^r \subset \GG$ denote the finite subset of rooted graphs of radius at most $r$, then $$
1 = \|f\| = \sum_{F\in \GG^r} \nu^\ast\{F\} \|f(F,\cdot)\|_2^2,
$$
where $\nu^\ast$ denotes the degree-biased version of the law of $(G,o)$. It follows that for every $F\in \GG^r$ such that $
\nu^* \{(G,o) : (B_G (o,r),o) \cong F\}\neq 0
$ we have $$
\|f(F,\cdot)\|_2 \leq \nu^* \{(G,o) : (B_G (o,r),o) \cong F\}^{-\frac{1}{2}}.
$$
If we let $K$ denote the maximum of the right hand side as $F$ ranges over $\GG^r$, we get that Equation (\ref{eq:sec.6.eq.1}) is bounded above by $
K^2 \EE_{\nu^\ast} [\PP[X_n \in B_{2r} (o)] ].
$

Combining all of the above we deduce that $$
\sqrt[n]{\langle \cal M^n f, f\rangle} \leq K^{\frac{2}{n}} \sqrt[n]{\EE_{\nu^\ast} [\PP[X_n \in B_r (o)] ]}.
$$
For some larger constant $L >0$ we obtain the bound \begin{align*}
\EE_{\nu^\ast} \left[\sum_{v\in B_r (o)} \PP [X_n =v]\right] &\leq   \EE_{\nu^\ast} \left[p_n (o,o) + \sum_{v\in B_r (o)\backslash o} \frac{1}{p_{d(o,v)} (o,v)} p_{n+ d(o,v)} (o,o)\right]  \\
&\leq L \EE_{\nu^\ast } [p_n (o,o)],
\end{align*}
By Theorem \ref{thm:q.vs.a} we deduce that $$
\sqrt[n]{\langle \cal M^n f, f\rangle} \leq (1 + o(1)) \rho_a (G,o) = (1 +o(1))\rho(G,o),
$$
concluding the proof.
\end{proof}

As we mentioned in the introduction, Theorem~\ref{thm:intro.rel.ram} complements Theorem \ref{thm:spec.containment}, showing that the spectral radius is achieved in the containment $\sigma(G,o) \subset \sigma_R (\cal B)$. 

\begin{proof}[Proof of Theorem~\ref{thm:Bernoulli_spectral_radius_equality}]
    Theorem \ref{thm:spec.containment} implies that $\rho(G,o) \leq \rho_R (\cal B)$, and Theorem~\ref{thm:intro.rel.ram} shows $\rho_R (\cal B) \leq \rho(G,o)$.
\end{proof}

\section{A finite-to-infinite strengthening}\label{sec:fin.to.inf}

In this section, we prove Theorem~\ref{thm:qt.spec.eq}, which is a strengthening of Theorem~\ref{thm:intro.rel.ram} for a special class of unimodular random graphs. This endeavor is motivated by two observations. First, in Section \ref{subsubsec:ramanujan_graphs_and_graphings} we discussed how deep theorems for finite graphs (e.g.\ Friedman's theorem) can imply results in a limiting setting (e.g.\ Theorem~\ref{thm:BSZV}). Second, Bernoulli graphings are reminiscent of random lifts of finite graphs, while the relative nature of Theorem~\ref{thm:intro.rel.ram} is reminiscent of new eigenvalues of such lifts. This leads naturally to the results of Bordenave and Collins recalled below.

Apart from the formal strengthening, our goal here is also to emphasize the parallels between the finite and infinite theory, with the hope that the two can aid and influence each other in future research. See also Section \ref{sec:further_directions} for further directions.

\subsection{New eigenvalues of random lifts}

As a first example, let us fix a finite base graph $H$, and build an $n$-fold random cover $H_n$ as follows. We set $[n]:=\{1,\dots,n\}$, let $V(H_n) = V(H)\times[n]$, and randomly add edges $e=(u,v) \in E(H_n)$ by picking independent uniform random bijections between $\{u\}\times[n]$ and $\{v\}\times[n]$, placing edges in $H_n$ accordingly.  

Recall that a graph homomorphism is a \emph{covering} if the neighbourhood of every vertex is bijectively mapped onto its image.We write $p_n \colon H_n \to H$ to denote the covering graph homomorphism defined by $p_n (u,i) = u$ for any $(u,i) \in V(H) \times [n]$.  Therefore $p_n$ induces an embedding $(p_n)^*: \ell^2(H,\deg) \to \ell^2(H_n,\deg)$ that is equivariant with respect to the Markov operators. Hence $\sigma(H) \subset \sigma(H_n)$ as multisets, and we define the multiset of \emph{new eigenvalues} of $H_n$ as $\sigma_N(H_n) = \sigma(H_n) \setminus \sigma(H)$. The relative version of Friedman's theorem, which is a special case of the more general result we will recall shortly, states that $\sigma_N(H_n) \subseteq [-\rho(\tilde{H})-\varepsilon,\rho(\tilde{H})+\varepsilon]$ with high probability for any $\varepsilon>0$. Here $\tilde{H}$ denotes the universal cover of $H$, which also happens to be the local limit of the $H_n$. Note that $\tilde{H}$ is always a tree, and if $H$ is a bouquet of $d$ loops, then $\tilde{H}=T_d$, and one recovers the original Friedman theorem.

We now introduce the more general model studied in \cite{bordenave2020new}, and state the result that we will apply. Intuitively, the independence of the choices of bijections in the random lifts considered above ``opened up'' the overwhelming proportion of cycles of $H$, this is why the limit is a tree. In general, instead of picking all bijections of the random cover independently, we will allow dependencies, like using the same bijection above two edges of $H$, or using the product of the bijections of two edges at a third one. This way, some cycles of $H$ might not open up.

More precisely, let $\FF_d = \langle s_1, \dots, s_d\rangle$ denote a free group on $d\in \NN$ generators. If $H$ is a graph, we let $\vec{E} (H)$ denote the set of oriented edges of $H$ and say that a map $\phi \colon \vec{E}(H) \to \FF_d$ is \emph{homomorphism-like} if $\phi_{(u,v)} = (\phi_{(v,u)})^{-1}$ for every $(u,v) \in \vec{E} (H)$. We use $\Sym ([n])$ to denote the set of permutations on $[n]$. If $g \in \FF_d$ and $S_1,\dots,S_d \in \Sym([n])$, we let $g[S_1,\dots,S_d] \in \Sym([n])$ denote the image of $g$ under the homomorphism $\FF_d \to \Sym ([n])$ defined by $   s_i \in  \FF_d  \mapsto S_i \in \Sym([n])$.

\begin{definition}\label{def:BC.lift}
    Let $H$ be a finite graph and $\phi \colon \vec{E}(H) \to \FF_d$ a homomorphism-like map. The $n$-th $\phi$\emph{-random lift} of $H$ is the random graph $H_n$ defined on the vertex set $V(H) \times [n]$ as follows. First, sample $S_1, \dots, S_d \in \Sym ([n])$ independently at random from the uniform distribution on $\Sym ([n])$. Then, for each $e = (u,v) \in \vec{E}(H)$ add an edge in $E(H_n)$ connecting $(u,i) \in V(H_n)$ with $(v, \phi_{(u,v)} [S_1, \dots,S_d] (i))$. 
\end{definition}

Any homomorphism-like map extends to a homomorphism $\overline \phi \colon \pi_1 (H) \to \FF_d$, where $\pi_1 (H)$ is the fundamental group of $H$, as follows. Fix a basepoint $v_0 \in V(H)$ and let $C$ be a returning walk on $H$, starting and finishing at $v_0$. Denote by $(e_1,\dots,e_n)$ its edge sequence and let $\phi' (C) = \phi_{e_1}\dots \phi_{e_n}$. It is easy to see that if $C$ is homotopy-equivalent to $C'$, then $C'$ may be obtained from $C$ by adding or deleting length two subpaths arising from traversing an edge back and forth. Since $\phi$ is homomorphism-like, we have that $\phi' (C) = \phi'(C')$. It thus follows that for any $[C] \in \pi_1 (H)$, the map $\overline\phi ([C]) := \phi' (C)$ is a well-defined homomorphism. 

It is straightforward to see that the $\phi$-random lifts $(H_n)_n$ converge locally to $G:= \tilde H / \ker (\overline\phi)$ with probability 1. The next theorem establishes a much stronger form of convergence.

\begin{theorem}[Bordenave--Collins \cite{bordenave2019eigenvalues}]\label{thm:BC}
Let $H$ be a finite graph and $\phi\colon \vec{E}(H) \to \FF_d$ a homomorphism-like map. If $(H_n)_n$ are $\varphi$-random lifts, and $G= \tilde H / \ker (\overline\phi)$, then $\sigma (H_n) \to \sigma (G)$ in the Haussdorff distance in probability.
\end{theorem}

Recall that the \emph{Haussdorff distance} between two closed susbets $U,V \subset \RR$ is defined as $$
d_H (U,V) = \inf\{\eps > 0 : U \subset V + [-\eps,\eps]\ \text{and}\ V \subset U + [-\eps,\eps]\}.
$$
Also, if $(M,d)$ is a metric space, a sequence of $M$-valued random variables $(X_n)_n$ converges \emph{in probability} to an $M$-valued random variable $Y$ if $\PP [ d(X_n,Y) \geq \eps] \to 0$ for every $\eps > 0$.

\subsection{Characterization of the limit objects}

Before exploiting Theorem \ref{thm:BC} in a limiting setting, we prove Theorem~\ref{thm:char.intro}, that is, we characterize the graphs $G$ that appear in it as local limits as unimodular quasi-transitive quasi-trees. Recall that a connected bounded degree graph $G$ is \emph{quasi-transitive} if its automorphism group $\Aut (G)$ acts on $G$ with finitely many orbits. Such $G$ is said to be \emph{unimodular} when $\Aut (G)$ is unimodular as a locally compact second countable group. This definition agrees with our previous one; indeed, a quasi-transitive graph $G$ is unimodular in this sense if and only if there exists a unimodular random graph $(H,o)$ such that $H$ is isomorphic to $G$ almost surely.  Finally, recall that $G$ is a \emph{quasi-tree} if it is quasi-isometric to a tree.

We recall that the conditions whose equivalence we aim to prove are the following:

\begin{enumerate}[(i)] 
    \item \label{itm:unimod_qt_qt} $G$ is a unimodular quasi-transitive quasi-tree.

    \item \label{itm:free_qt_action_of_Fd} There exists a free quasi-transitive action $\FF_d \actson G$.
  
    \item \label{itm:hom_like_map} There exists a finite graph $H$ and a homomorphism-like map $\phi \colon \vec{E}(H) \to \FF_{d'}$  such that $G\cong \tilde H / \ker (\overline\phi)$.
\end{enumerate}

We were not able to find this characterization in the literature, as most sources did not include unimodularity as part of the topics of investigation.

\begin{proof}[Proof of Theorem~\ref{thm:char.intro}]
That (\ref{itm:hom_like_map}) implies (\ref{itm:free_qt_action_of_Fd}) follows from the fact that the group of deck transformations of the covering $G \to H$ is $\pi_1 (H) / \ker (\overline\phi)$, see for instance \cite[Proposition 1.39]{hatcher}. By the first isomorphism theorem, the latter is isomorphic to $\im (\overline\phi) \leq \FF_{d'}$, which is a finitely generated -- since $\pi_1 (H)$ is finitely generated -- free group by the Nielsen-Schreier Theorem.

\smallskip
Conversely, suppose that (\ref{itm:free_qt_action_of_Fd}) holds, aiming to prove (\ref{itm:hom_like_map}). Let $H = G/\FF_d$, and let $p \colon G \to H$ denote the corresponding covering map. The map $p$ is moreover a normal covering, as free actions on graphs are properly discontinuous, and so $p_\ast \colon \pi_1 (G) \hookrightarrow \pi_1 (H)$ with $p_\ast (\pi_1 (G)) \trianglelefteq \pi_1 (H)$ and $\pi_1 (H) / p_\ast (\pi_1 (G)) \cong \FF_d$. We set $d'=d$, and construct a homomorphism-like map $\phi \colon \vec{E}(H) \to \FF_d$ as follows. Let $T$ be a spanning tree for $H$. For every edge in $T$ declare its $\phi$ image to be the identity. For an oriented edge $e=(u,v)$ not in $T$ let $g_e$ denote the corresponding generator of $\pi_1(H)$, that is, the (homotopy class of the) returning walk $(w(o,u),e,w(v,o))$, where $w(o,x)$ is the unique oriented path from $o$ to $x$ in $T$. For such an edge $e$ let $\phi(e) = \psi(g_e)$ where $\psi$ denotes the quotient map $\pi_1 (H) \to\pi_1 (H) / p_\ast (\pi_1 (G)) \cong \FF_d$. Clearly we obtain a homomorphism-like map $\phi$ such that $\ker (\overline \phi) = p_\ast (\pi_1 (G))$, and therefore $\tilde H / \ker (\overline\phi) \cong G$. 

\smallskip
Let us now show that (\ref{itm:free_qt_action_of_Fd}) implies (\ref{itm:unimod_qt_qt}). On the one hand, it is easy to see that $G$ is quasi-isometric to $\FF_d$, hence it is a quasi-tree. On the other hand, it follows from \cite[Exercise 8.8]{LP} that $\Aut (G)$ is unimodular quasi-transitive, as it contains a quasi-transitive subgroup isomorphic to the unimodular $\FF_d$.

\smallskip
Lastly, we show that (\ref{itm:unimod_qt_qt}) implies (\ref{itm:free_qt_action_of_Fd}).  By \cite[Theorem 7.4]{Stallings} the graph $G$ admits an $\Aut (G)$-invariant tree-decomposition $(T,\cal V)$ such that $\Aut (G)$ acts on the edges of $T$ quasi-transitively and the elements of $\cal V$ have at most $K\in \NN$ many elements. Let $H$ be the image of $\Aut (G)$ in $\Aut (T)$, and write $\psi: \Aut (G) \to H$ for the homomorphism.

First, we argue that $H$ is unimodular. Indeed, $H \cong \Aut(G)/\ker(\psi)$, and $\ker(\psi)$ contains those automorphisms of $G$ that set-wise fix all parts of the tree-decomposition. Notice that $\ker(\psi)$ is a compact subgroup, implying that the quotient $H$ is unimodular. 

Second, by \cite[Corollary 4.8]{BK} the group $H$ contains a free group $\FF_d:= \langle s_1, \dots, s_d\rangle$ acting freely quasi-transitively on $T$. (Acting without inversions can be assumed upon passing to a barycentric subdivision, see \cite[Proposition 6.5]{Bass}). Let $\tilde s_1, \dots, \tilde s_n \in \Aut (G)$ be such that $\psi(\tilde s_i) = s_i$ for each $i$. It is easy to see that $\Gamma := \langle \tilde s_1, \dots \tilde s_d\rangle \leq \Aut (G)$ is a free group. Moreover, we claim that $\Gamma$ acts on $G$ quasi-transitively. Therefore, if $\FF_d$ acts on $T$ with $N$ many orbits, there can be at most $KN$ $\FF_d$-orbits on $G$. Since every vertex is in finitely many parts of the tree decomposition and $\FF_d$ acts freely on $T$, the stabilizers of the action of $\Gamma$ on $G$ are finite. However, as $\Gamma$ is a free group, this implies that stabilizers are trivial.
\end{proof}

\subsection{Quasi-random colorings}

The last ingredient of the proof of Theorem~\ref{thm:qt.spec.eq} is a quasi-random coloring of the $\phi$-random lifts. Let $(H_n)_n$ be $\varphi$-random lifts of $H$, and $(G,o)$ their Benjamini-Schramm limit. Note that by quasi-transitivity $(G,o)$ takes finitely many possible values: for $h\in V(H)$ and any $v, v' \in V(G)$ that are mapped to $h$ under the covering map $G \to H$, $(G,v) \cong (G,v')$ as deterministic rooted graphs. By slight abuse of notation, we write $(G,h)$ for this deterministic graph.

Let $(G,o,c)$ and $(G,h,c)$ be the $[k]$-Bernoulli colouring of $(G,o)$ and $(G,h)$ respectively. That is, in $(G,h,c)$ the rooted graph is the same with probability 1, but the coloring is iid uniform random. In $(G,o,c)$ the rooted graph is also random, in fact $(G,o,c)$ is the convex combination of the $(G,h,c)$'s as $h$ varies over $H$. By a standard probabilistic construction, we can build \emph{deterministic} colorings of the $H_n$ that model each $(G,h,c)$ if sampled uniformly from $\{h\}\times[n]$ (that is, the preimage of $h$).

\begin{lemma}\label{lemma:quasi.random}
     With probability 1 the random graphs $H_n$, admit $[k]$-colourings $c_n: V(H_n) \to [k]$ such that 
     for any $h \in V(H)$, by picking a uniform random vertex $v$ from $p_n^{-1}(h)$ we have $(H_n,v,c_n) \to (G,h,c)$ in distribution. In particular,  $(H_n,v_n) \to (G,o,c)$ in the coloured Benjamini-Schramm sense.
\end{lemma}

\begin{proof}
    We generate a realization of $H_n$. With probability 1, the $H_n$ Benjamini-Schramm converge to $(G,o)$, moreover, this convergence holds individually for all $(G,h)$ when sampling the root of $H_n$ among those covering $h$. That is, if $v$ is as in the statement of the lemma, $(H_n, v) \to (G,h)$ in distribution. As $(G,h)$ is deterministic, this simply means that for any radius $r$ the proportion of vertices in $p_n^{-1}(h)$ whose $r$-neighborhood is isomorphic to $B_G(r,h)$ converges to 1.
    
    Having fixed the graphs, we pick the colours $c_n(u)$ independently and uniformly randomly from $[k]$ for all $u \in V(H_n)$. It is not immediately obvious, but a standard argument (using the convergence property from the previous paragraph) to check that the conclusions of the lemma hold for $c_n$ with probability 1, that is, for almost every realization of $c_n$. See \cite[Lemma 7.9]{HLS} for a similar argument.  
\end{proof}



    



    

\subsection{Spectra of quasi-transitive quasi-trees}

We can now address the proof of the main theorem of the section. Let $G$ be a unimodular quasi-transitive quasi-tree. We denote the associated unimodular random graph by $(G,o)$ and its Bernoulli graphing by $(\cal B,\GG^+,\mu^\ast)$. We aim to show $\sigma_R (\cal B) = \sigma (G)$.

\begin{proof}[Proof of Theorem \ref{thm:qt.spec.eq}]
    By Theorem \ref{thm:spec.containment} we have that $\sigma (G) \subset \sigma_R (\cal B)$. Suppose for the sake of contradiction that this inclusion is strict, so there exists $\lambda \in \sigma_R (\cal B) \backslash \sigma (G)$. Since $\sigma (G)$ is closed, there exists an $\eps > 0$ such that $(\lambda + [-\eps,\eps] )\cap (\sigma (G) + [-\eps,\eps]) = \emptyset$. 

By Theorem \ref{thm:char.intro}, there exists a finite graph $H$ and a homomorphism-like $\phi$ such that $G \cong \tilde H / \ker (\overline{\phi})$. Let $(H_n)_n$ be the sequence of $\phi$-random lifts. Theorem~\ref{thm:BC} states that $$
    \PP[\sigma (H_n) \subset \sigma (G) + [-\eps,\eps]] \to 1.
    $$ 
In particular, this implies that $M_n - \lambda$ is bounded below by $\eps$ with high probability, where $M_n$ denotes the Markov operator of $H_n$ on $\ell^2 (H_n,\deg)$. Our aim is to contradict this. 

    For the remainder of the proof, we assume, again by the Isomorphism Theorem for standard probability spaces \cite[Theorem 17.41]{kech}, that the Bernoulli random labels are of the form $\omega \in 2^\NN$, with each $\omega_i\in \{0,1\}$ sampled uniformly and independently at random. Under this assumption, we say that $f\in L^2 (\GG^+,\mu^\ast)$ has complexity at most $k\in \NN$ if for any two $(G,\xi,o)$ and $(G,\eta,o)$ we have that $$
    \xi|_{[k]}= \eta'|_{[k]} \Rightarrow f(G,\xi,o) = f(G,\eta,o).
    $$
    Let $f \in  R$  
    be a normalised block factor of radius $r$ and complexity at most $k$ such that $\|(M - \lambda) f\|< \eps$. Such $f$ exists by hypothesis and density of block factors in $R$, see Corollary \ref{cor:proj.block}. Let $\underline f \colon \GG^{r,k} \to \RR$ be the map defining the block factor $f$, where $\GG^{r,k}$ is the finite set of rooted graphs of radius at most $r$, labelled by sequences in $2^k$. 
    
    If we let $(c_n)_n$ be the colourings of Lemma \ref{lemma:quasi.random}, then we may let $g_n := \underline f \circ s_n$, where $s_n$ is the sampling map on $H_n$. We then have that $$
    \|g_n\| = \|\underline f\|_{L^2 (\GG^+, \mu_n)} \to   \|\underline f\|_{L^2 (\GG^+, \mu^\ast_{r,k})} = 1 
    $$
    where $\mu_n$ is the pushforward by $s_n$ of the counting measure on the $H_n$, and $\mu^\ast_{r,k}$ the corresponding for $\mu^\ast$. A similar argument shows that $$
    \|(M - \lambda) g_n\| \to \|(M - \lambda) f\|. 
    $$

    In order to conclude the proof, we only need to show that the $g_n$ may be assumed to be in the orthogonal to $H$, up to small editions. However, by Lemma \ref{lemma:quasi.random}, we have that for every $h\in V(H)$, letting $\tilde h_n$ denote its pre-image by the $n$-th covering map and $\mu_n^h$ the distribution on $\GG$ by uniformly sampling on $\tilde h_n$, $$
    \langle g_n , \one_{\tilde h _n}\rangle = \EE_{\mu_n^h}[ \underline f ] \to \EE_\alpha [f ] = 0,
    $$
    for some fixed type $\alpha \in \GG$. This $\alpha$ is determined by the pre-images of $h$ in the universal cover. By the characterization Corollary \ref{cor:random.char} of the random spectrum the conclusion follows. Since the subspace generated by the vertices of $H$ has dimension $|V(H)|$, the latter inequality implies that $\|g_n - P_ng_n\| \to 0$, where $P_n$ is the projection on the orthogonal of $H$ in $\ell^2 (H_n,\deg)$. 

    We obtain a contradiction with Theorem \ref{thm:BC} for $n$ sufficiently large.
\end{proof}

\section{Further directions} \label{sec:further_directions}

It is natural to ask if Friedman's theorem can be extended to large random graphs that are not regular, but have a fixed degree distribution. One can generate such graphs by the \emph{configuration model} as follows. Fix a degree distribution $(p_i)_{i=3}^{\infty}$. (We assume $i \geq 3$ to exclude models without a spectral gap.) Build a graph on $n$ vertices by attaching $i$ half-edges to $p_i$ proportion of them. Assume for simplicity that $p_in \in \mathbb{N}$ for all $i$. Pick a uniform random matching on the half-edges, and connect the pairs to form actual edges. 

\begin{question} \label{qst:irreg_friedman}
    Are random graphs generated by the configuration model with fixed degree distribution almost Ramanujan with high probability?
\end{question}

Dropping the regularity assumption might seem like a technicality, but in fact it introduces a conceptual challenge, as the universal cover of the finite graphs is not a fixed, deterministic object anymore. The Benjamini-Schramm limit is a unimdular Galton-Watson tree $(T,o)$, where the root has degree distribution $(p_i)$, and every later vertex has $(i-1)$ children with probability $q_{i-1}=\frac{ip_i}{\sum i p_i}$. It is not known if the Bernoulli graphing of $(T,o)$ is Ramanujan, even though this should be a significantly easier question.

\begin{question} \label{qst:UGW_ramanujan}
    Is the Bernoulli graphing $\mathcal{T}$ of a unimodular Galton-Watson tree $(T,o)$ Ramanujan? That is, do we have $\rho(\mathcal{T}) = \rho(T,o)$?
\end{question}

A positive answer to Question~\ref{qst:irreg_friedman} would imply the same for Question~\ref{qst:UGW_ramanujan}. By Theorem~\ref{thm:intro.rel.ram}, the latter question reduces to $\rho_S(\mathcal{T}) \leq \rho(T,o)$. Note also that $\rho(T,o)$ only depends on the minimum degree $d_{\min}=\min\{i \mid p_i>0\}$, in fact $\rho(T,o)=\rho(T_{d_{\min}})$.

In terms of relating the local and global spectrum, we do not know if Theorem~\ref{thm:qt.spec.eq} holds in general.

\begin{question}
    Does $\sigma_R(\mathcal{G}) = \sigma(G,o)$ hold for all unimodular random graphs $(G,o)$?
\end{question}

Finally, we believe the connection to finite random graphs raises broader questions worth investigating. A few are listed here as examples. 

\begin{itemize}
    \item What other URGs appear as limits of finite random graph models where relative Ramanujan statements have been established? (For example, one would expect \cite{magee2025strong}  to imply $\rho(\mathcal{G})=\rho(G,o)$ for Cayley graphs of surface groups.) 
    
    \item Can one prove relative Ramanujan statements for random graph models where the limits are non-atomic URGs? (The irregular configuration model above would be such an example.)
    
    \item Does the strong convergence of random representations (as  established in \cite{bordenave2019eigenvalues, magee2025strong}) provide more information on the spectrum of the limit?
\end{itemize}


\def\MR#1{}
\bibliographystyle{amsalpha} 
\bibliography{bib}

\end{document}